\documentclass[12pt,reqno,oneside]{amsart}

\usepackage{hyperref}

\usepackage{comment}

\usepackage{graphicx}

\usepackage{amsmath,amsfonts,amsthm,mathrsfs,amssymb,cite}
\usepackage[usenames]{color}
\usepackage{enumerate}

\usepackage[left=2.3cm,right=2.3cm,top=3cm,bottom=3cm]{geometry}

\newcommand{\R}{\ensuremath{\mathbb{R}}} 
\newcommand{\C}{\ensuremath{\mathbb{C}}} 
\newcommand{\dx}[1][x]{\ensuremath{\,{\rm{d}} #1}}
\DeclareMathOperator{\di}{div}
\DeclareMathOperator{\cu}{curl}
\newcommand{\norm}[1]{\lVert#1\rVert}

\newcommand{\kommentar}[1]{}

\newtheorem{thm}{Theorem}[section]

\newtheorem{lem}{Lemma}[section]

\theoremstyle{definition}
\newtheorem{defn}{Definition}[section]
\theoremstyle{remark}

\newtheorem{rem}{Remark}[section]
\numberwithin{equation}{section}

\title[Localizing electromagnetic fields]{On localizing and concentrating electromagnetic fields}

\author[Harrach]{Bastian Harrach}
\address{Institute for Mathematics, Goethe University Frankfurt, Germany}
\email{harrach@math.uni-frankfurt.de}

\author[Lin]{Yi-Hsuan Lin}
\address{Institute for Advanced Study, The Hong Kong University Science and Technology, Clear Water Bay, Kowloon, Hong Kong}
\email{yihsuanlin3@gmail.com}

\author[Liu]{Hongyu Liu}
\address{Department of Mathematics, Hong Kong Baptist University, Kowloon, Hong Kong SAR, China}
\email{hongyu.liuip@gmail.com; hongyuliu@hkbu.edu.hk}

\begin{document}
	
\begin{abstract}
		We consider field localizing and concentration of electromagnetic
		waves governed by the time-harmonic anisotropic Maxwell system in a bounded domain.
		It is shown that there always exist certain boundary inputs which can generate
		electromagnetic fields with energy localized/concentrated in a given
		subdomain while nearly vanishing in another given subdomain. The theoretical results may have potential applications
		in telecommunication, inductive charging and medical therapy. 
		We also derive a related Runge approximation result for the time-harmonic anisotropic Maxwell system with partial boundary data.
		
		\medskip
		
		
	\end{abstract}
	
	\maketitle

	\tableofcontents

\section{Introduction}

\subsection{Background and motivation}

The electromagnetic (EM) phenomena are ubiquitous and they lie at
the heart of many scientific and technological applications including
radar and sonar, geophysical exploration, medical imaging, information
processing and communication. In this paper, we are mainly concerned with the mathematical
study of field localizing and concentration of electromagnetic waves
governed by the time-harmonic Maxwell system in a bounded anisotropic medium.
More specifically, we show that there always exist certain boundary
inputs which can generate the desired electromagnetic fields that are localized/concentrated
in a given subdomain while nearly vanishing in another given subdomain.

The localizing and concentration of electromagnetic fields can have many potential
applications. In telecommunication \cite{Wiki1}, one common means of
transmitting information between communication participants is via
the electromagnetic radiation. In a certain practical scenario, say
the secure communication, one may intend the information to be transmitted
mainly to a partner located at a certain region, while avoid the transmission
to another region. Clearly, if the information is encoded into the
electromagnetic waves that are localized and concentrated in the region
where the partner is located while are nearly vanishing in the undesired
region, then one can achieve the expected telecommunication effect. In the
setup of our proposed study, one can easily obtain the aforementioned
communication effect, in particular if the communication participants
transmit and receive information on some surface patches.

Concentrating electromagnetic fields can also be useful in inductive
charging, also known as wireless charging or cordless charging \cite{Wiki2},
which is an emerging technology that can have significant impact on
the real life. It uses electromagnetic fields to transfer energy between
two objects through electromagnetic induction. Clearly, the energy
transfer would be more efficient and effective if the corresponding
electromagnetic fields are concentrated around the charging station.
The localizing of electromagnetic fields can also have potential
application in electromagnetic therapy. Though it is mainly considered
to be pseudoscientific with no affirmative evidence, the electromagnetic
therapy has been widely practiced which claims to treat disease by
applying electromagnetic radiation to the body. If the electromagnetic
therapy shall be proven to be effective, then through the use of certain
purposely designed sources, one can generate electromagnetic fields
that are concentrated around the diseased area.

The above conceptual and potential applications make the study of
field concentration and localizing much appealing. Nevertheless, it
is emphasized that in the current article, we are mainly concerned
with the mathematical and theoretical study. We achieve some substantial
progress on this interesting topic, though the corresponding study
is by no means complete. It is also interesting to note that the localizing
of resonant electromagnetic fields has been used to produce invisibility
cloaking and has received significant attention in the literature in recent years \cite{ammari2013spectral,ando2016plasmon,bouchitte2010cloaking,li2015quasi,milton2006cloaking,nicorovici2007quasistatic}.
The corresponding study is mainly based on the use of plasmonic materials
to induce the so-called anomalous localized resonance.

Our mathematical argument for proving the existence of the localized and concentrated electromagnetic fields is mainly based on combining the unique continuation property for the anisotropic Maxwell system with a functional analytic duality argument developed in \cite{gebauer2008localized}.
By a similar argument, we also obtain a related Runge approximation property.
 
The use of blow up solutions has a long tradition in the study of inverse boundary value problems, cf.\ \cite{alessandrini1990singular,isakov1988uniqueness,kohn1984determining,kohn1985determining} for early seminal works on this topic.
Moreover, the combination of localized fields and monotonicity relations have led to the development of
monotonicity-based methods for obstacle/inclusion detection, cf.\ \cite{tamburrino2002new,harrach2013monotonicity}
for the origins and mathematical justification of this approach,  
\cite{barth2017detecting,brander2017monotonicity,garde2017comparison,garde2017convergence,garde2017regularized,harrach2010exact,harrach2015combining,harrach2017nonlocal,harrach2016enhancing,harrach2016monotonicity,harrach2015resolution,maffucci2016novel,su2017monotonicity,tamburrino2016monotonicity,ventre2017design,zhou2017monotonicity}
for further recent contributions, and the recent works \cite{harrach2017monotonicity,griesmaier2018monotonicity} for the Helmholtz equation. Theoretical uniqueness results for inverse coefficient problems have also been obtained by this approach in \cite{arnold2013unique,harrach2009uniqueness,harrach2012simultaneous,harrach2010exact,harrach2017local}). 

In this work, we show the existence of localized electromagnetic fields for the more challenging case of time-harmonic anisotropic Maxwell system
with partial data. We also derive a Runge approximation result, that shows that every solution in a subdomain can be approximately well
by a solution on the whole domain. In that context let us note the famous equivalence theorem from Peter Lax \cite{lax1956stability}: the weak unique continuation property is equivalent to the Runge approximation property for the second order elliptic equation. In our study, we affirmatively verify this property still holds for the anisotropic Maxwell system.

The rest of the present section is devoted to the mathematical description of the setup of our study and the statement of the main result.

\subsection{Mathematical setup and statement of the main result }

Let $\Omega$ be a simply connected domain in $\mathbb{R}^{3}$
with a Lipschitz connected boundary $\partial\Omega$. Let $\epsilon=(\epsilon_{ij})_{1\leq i,j\leq 3}$ and
$\mu=(\mu_{ij})_{1\leq i,j\leq 3}$ be two $3\times 3$ real matrix-valued functions satisfying

\begin{itemize}
\item  Strong ellipticity: There exist constants $\mu_{0}>0$
and $\epsilon_{0}>0$ verifying 
\begin{equation}
\begin{cases}
\mu_{0}|\xi |^2\leq\sum_{i,j=1}^3\mu_{ij}(x)\xi_i\xi _j\leq\mu_{0}^{-1}|\xi|^2,\\
\epsilon_{0}|\xi|^2\leq\sum_{i,j=1}^3\epsilon_{ij}(x)\xi_i\xi_j\leq\epsilon_{0}^{-1}|\xi|^2,\label{Positive lower bound}
\end{cases}
\mbox{ for any }x\in\Omega \text{ and }\xi \in \mathbb R^3.
\end{equation}

\item Smoothness: $\epsilon$ and $\mu$ are piecewise Lipschitz continuous matrix-valued functions defined in $\Omega$.

\item Symmetry: $\epsilon$ and $\mu$ are symmetric matrices, that is, $\epsilon_{ij}=\epsilon_{ji}$ and $\mu_{ij}=\mu_{ji}$ for all $1\leq i,j\leq 3$.

\end{itemize}
 
The functions $\epsilon$ and $\mu$, respectively, signify the electric permittivity and magnetic permeability of the medium in $\Omega$. Consider the time-harmonic electromagnetic wave propagation in $\Omega$. \textcolor{black}{With the $e^{ikt}$ time-harmonic convention assumed, we let $E(x)$ and $H(x)$, respectively, denote the electric and magnetic fields. Here, $k\in\mathbb{R}_+$ signifies a wavenumber. Then the  
electromagnetic wave propagation is governed by the following Maxwell system, 
\begin{equation}
\begin{cases}
\nabla\times E-ik\mu H=0 & \mbox{ in }\Omega,\\
\nabla\times H+ik\epsilon E=0 & \mbox{ in }\Omega,\\
\nu\times E=\begin{cases}
f & \mbox{ on }\Gamma\\
0 & \mbox{ otherwise}
\end{cases} & \mbox{ on }\partial\Omega,
\end{cases}\label{Maxwell system}
\end{equation}
where $\Gamma$ is an arbitrary nonempty relative open subset of $\partial\Omega$
and $\nu$ is the unit outer normal vector on $\partial\Omega$. }
\textcolor{black}{It is assumed that $k>0$ is not an eigenvalue (or non-resonance,
see Section \ref{Section 2}) for \eqref{Maxwell system}} and $f\in C^\infty_c(\Gamma)$ throughout this paper. 

The main result concerning the localized electromagnetic fields for the anisotropic Maxwell system \eqref{Maxwell system} is contained in the following theorem. 

\begin{thm}
\label{Main Theorem}
Let $\Omega\subset\mathbb{R}^{3}$ be a bounded Lipschitz domain and $\Gamma\subseteq\partial \Omega$ be
a relatively open subset of the boundary. Let $\epsilon,\mu\in L^\infty(\Omega,\R^{3\times 3})$ be real-valued, piecewise Lipschitz continuous functions satisfying \eqref{Positive lower bound} and $k>0$ be
a non-resonant wavenumber. Let $D\Subset\Omega$ be a closed set with connected complement $\Omega \setminus D$.
For every open set $M\subseteq \Omega$ with $M\not\subseteq D$ (see Figure \ref{fig:1} for the schematic illustration), 
there exists a sequence $\left\{ f^{(\ell)}\right\} _{\ell\in\mathbb{N}}\subset C_c^\infty(\Gamma)$
such that the electromagnetic fields fulfill
\begin{equation*}
\int_{M}\left(|E^{(\ell)}|^{2}+|H^{(\ell)}|^2\right)\dx\to\infty\quad \mbox{ and }\quad \int_{D}\left(|E^{(\ell)}|^{2}+|H^{(\ell)}|^2\right)\dx\to0\mbox{ as }\ell\to\infty,
\end{equation*}
where, for $\ell\in\mathbb{N}$, $(E^{(\ell)},H^{(\ell)})\in L^2(\Omega)^3\times L^2(\Omega)^3$
is a solution of 
\begin{align*}
\begin{cases}
\nabla\times E^{(\ell)}-ik\mu H^{(\ell)}=0   & \mbox{ in }\Omega,\\
\nabla\times H^{(\ell)}+ik\epsilon E^{(\ell)}=0  & \mbox{ in }\Omega,
\end{cases}
\end{align*}
with boundary data
\begin{align*}
\nu\times E^{(\ell)}|_{\partial \Omega}=\begin{cases}
f^{(\ell)} & \mbox{ on }\Gamma,\\
0 & \mbox{ otherwise.}
\end{cases}
\end{align*}
\end{thm}
\begin{rem}
We call the sequence $\{(E^{(\ell)},H^{(\ell)})\}_{\ell \in \mathbb N}$ in Theorem \ref{Main Theorem} to be the localized electromagnetic fields.
\end{rem}
\begin{figure}[t]
\centering {\includegraphics[width=6cm]{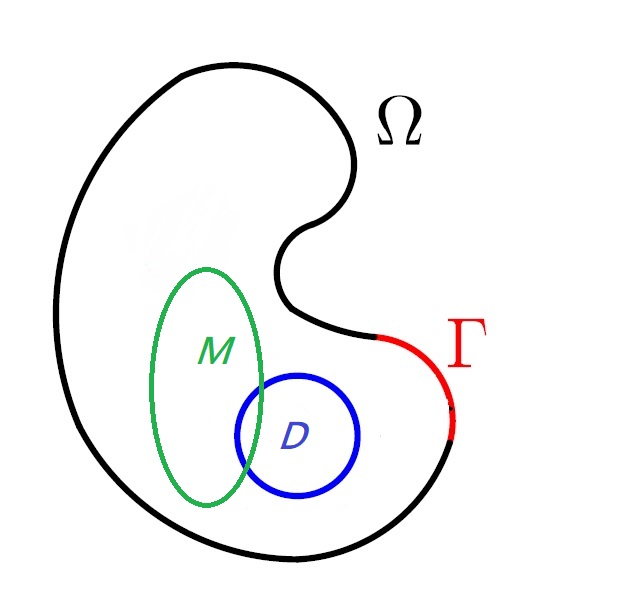}} \protect\caption{Schematic illustration of the field localizing and concentration. }
\label{fig:1} 
\end{figure}

The rest of the paper is organized as follows. In Section \ref{Section 2}, we
present results on the well-posedness of the time-harmonic anisotropic Maxwell system. We also provide the unique continuation property (UCP) for the anisotropic Maxwell system, whenever the coefficients $\mu$ and $\epsilon$ are piecewise Lipschitz continuous matrix-valued functions. In Section \ref{Section Localized fields}, we demonstrate that there exist localized electromagnetic fields, which proves Theorem \ref{Main Theorem}. The method relies on functional analysis
techniques. In Section \ref{Section 4} we prove a related Runge approximation property for the anisotropic Maxwell system with parital boundary data.

\section{The anisotropic Maxwell system in a bounded domain \textcolor{black}{\label{Section 2}}}

In this section, we summarize some useful results of the Maxwell system, including unique solvability and a unique continuation property. Throughout this section let $\Omega\subset\mathbb{R}^{3}$ be a bounded Lipschitz domain. 

\subsection{Spaces and traces}\label{subsect:spaces_traces}
We introduce the spaces
\begin{align*}
H(\di,\Omega)&:=\left\{ E\in L^{2}(\Omega)^3;\quad\nabla\cdot E\in L^{2}(\Omega) \right\},\\
H(\mathrm{curl},\Omega)&:=\left\{ E\in L^{2}(\Omega)^3;\quad\nabla\times E\in L^{2}(\Omega)^3\right\},
\end{align*}
and the tangential trace operators 
\begin{align*}
\gamma_t:&\ H(\mathrm{curl},\Omega)\to H^{-1/2}(\mathrm{div}_{\partial \Omega},\partial \Omega ), \quad E\mapsto \gamma_t E:=\nu \times E|_{\partial \Omega},\\
\gamma_T:&\ H(\mathrm{curl},\Omega)\to H^{-1/2}(\mathrm{curl}_{\partial \Omega},\partial \Omega ), \quad E\mapsto \gamma_T E := \nu \times (E\times \nu)|_{\partial \Omega},
\end{align*}
where here and in the following all functions are complex-valued unless indicated otherwise.  
Then $\gamma_t$ and $\gamma_T$ are surjective bounded linear operators with bounded right inverses $\gamma_t^{-1}$ and
and $\gamma_T^{-1}$, the space $H^{-1/2}(\mathrm{div}_{\partial \Omega},\partial \Omega )$ can be identified with the dual of $H^{-1/2}(\mathrm{curl}_{\partial \Omega},\partial \Omega )$,
and for all $E,F\in H(\mathrm{curl},\Omega)$ we have the integration by parts formula
\begin{equation}
\label{eq:int_by_parts_tangential}
\int_\Omega (\nabla \times E) \cdot F \dx - \int_\Omega E \cdot (\nabla \times F) \dx
= \int_{\partial \Omega} \left(\nu \times E|_{\partial \Omega}\right) \cdot  \left(\nu \times (F|_{\partial \Omega}\times \nu)\right) \dx[S]
\end{equation}
(cf. \cite{buffa2002traces,M2003book}), where the dual pairing on $H^{-1/2}(\mathrm{div}_{\partial \Omega},\partial \Omega )\times H^{-1/2}(\mathrm{curl}_{\partial \Omega},\partial \Omega )$ is written 
as integral for the sake of readability.

The subspace of $H(\cu,\Omega)$-functions with vanishing tangential trace is denoted by
\begin{align*}
H_0(\cu, \Omega)&:=\{ E\in H(\cu,\Omega):\ \nu \times E|_{\partial \Omega}=0\}.
\end{align*}
$H_0(\cu, \Omega)$ is a closed subspace of $H(\cu,\Omega)$ and $C_0^\infty(\Omega)^3$ is dense in $H_0(\cu, \Omega)$, cf. \cite{M2003book}.

To treat partial boundary data on a relatively open subset $\Gamma\subseteq \partial \Omega$ , we also introduce 
the space of functions on $\Gamma$ that can be extended by zero to the trace of a  
$H(\mathrm{curl},\Omega)$-function
\begin{align}\label{abstract trace space}
H(\Gamma):=\text{closure of }C^\infty_c(\Gamma)\text{ in } H^{-1/2}(\mathrm{div}_{\partial \Omega},\partial \Omega ).
\end{align}
For all $E\in H(\mathrm{curl},\Omega)$, we identify the restricted trace 
$\nu \times (E\times \nu)|_{\Gamma}$ with the quotient space element
\[
\nu \times (E\times \nu)|_{\partial \Omega} + H(\Gamma)^\perp \in H^{-1/2}(\mathrm{curl}_{\partial \Omega},\partial \Omega ) / H(\Gamma)^\perp
 = H(\Gamma)^*,
\]
and thus define the restricted trace operator 
\[
\gamma_T^{(\Gamma)}:\ H(\mathrm{curl},\Omega)\to H(\Gamma)^*,\quad E\mapsto \gamma_T^{(\Gamma)} E := \nu \times (E\times \nu)|_{\Gamma}.
\]

\subsection{Well-posedness of the anisotropic Maxwell system}

Given anisotropic coefficients $\epsilon,\mu\in L^\infty(\Omega,\R^{3\times3})$ satisfying \eqref{Positive lower bound},
$k>0$, $J,K\in L^{2}(\Omega)^3$ and $f\in H^{-1/2}(\mathrm{div}_{\partial \Omega},\partial \Omega )$, we consider the 
Maxwell system for $(E,H)\in L^2(\Omega)^3\times L^2(\Omega)^3$
\begin{alignat}{2}
\label{eq:Maxwell_inhom1} \nabla\times E-ik\mu H&=K  && \quad \mbox{ in }\Omega,\\
\label{eq:Maxwell_inhom2} \nabla\times H+ik\epsilon E&=J  && \quad \mbox{ in }\Omega,\\
\label{eq:Maxwell_inhom3} \nu\times E|_{\partial \Omega}&=f.
\end{alignat}
Note that the boundary condition \eqref{eq:Maxwell_inhom3} is well-defined since 
\eqref{eq:Maxwell_inhom1} implies $E\in H(\cu,\Omega)$.

For the variational formulation of \eqref{eq:Maxwell_inhom1}--\eqref{eq:Maxwell_inhom2} we introduce the sesquilinear form
\begin{align*}
\mathcal B &:\ H(\cu,\Omega)\times H(\cu,\Omega)\to \C,\\
\mathcal B(E,F) &:=\int_\Omega \left( \mu^{-1}  \nabla \times E \right)
\cdot \left(\nabla \times \overline{F}\right) \dx 
 - \int_\Omega  k^2 \epsilon E \cdot \overline F \dx.
\end{align*}

Then we have the following variational formulation and well-posedness result.

\begin{thm}\label{thm:Maxwell_well_posedness}
\begin{enumerate}[(a)]
\item $(E,H)\in L^2(\Omega)^3\times L^2(\Omega)^3$ solve \eqref{eq:Maxwell_inhom1}--\eqref{eq:Maxwell_inhom2}
if and only if $E\in H(\cu,\Omega)$ solves
\begin{equation*}\label{eq:Maxwell_inhom_variational}
\mathcal B(E,F)= \int_\Omega ikJ\cdot \overline F\dx +  \int_\Omega  \left( \mu^{-1} K \right) \cdot \left(\nabla \times \overline F\right) \dx
\quad \text{for all $F\in H_0(\cu,\Omega)$,}
\end{equation*}
and $H= -\dfrac{i}{k} \mu^{-1} (\nabla \times E - K)$.
\item The set of $k>0$ for which the homogeneous system \eqref{eq:Maxwell_inhom1}--\eqref{eq:Maxwell_inhom3} with
$J=0$, $K=0$, and $f=0$, possesses a non-trivial solution is discrete. We call these $k$ \emph{resonance frequencies}.
\item If $k$ is not a resonance frequency, then there exists a unique solution $(E,H)\in L^2(\Omega)^3\times L^2(\Omega)^3$ of
\eqref{eq:Maxwell_inhom1}--\eqref{eq:Maxwell_inhom3}, and the solution depends linearly and continuously on $J,K\in L^2(\Omega)^3$ and
$f\in H^{-1/2}(\mathrm{div}_{\partial \Omega},\partial \Omega )$.
\end{enumerate}
\end{thm}

The proof of Theorem \ref{thm:Maxwell_well_posedness} follows standard arguments. Since we did not find a reference for precisely this setting, 
we give a proof for the sake of completeness. For the proof we will use the following lemmas.

\begin{lem}\label{lemma:Maxwell_variational_H0curl}
$(E,H)\in L^2(\Omega)^3\times L^2(\Omega)^3$ solve \eqref{eq:Maxwell_inhom1}--\eqref{eq:Maxwell_inhom2}
if and only if $E\in H(\cu,\Omega)$ solves
\begin{equation}\label{eq:Maxwell_inhom_variational_Hcurl}
\mathcal B(E,F)= \int_\Omega ikJ\cdot \overline F\dx +  \int_\Omega  \left( \mu^{-1} K \right) \cdot \left(\nabla \times \overline F\right) \dx
\quad \text{for all $F\in H_0(\cu,\Omega)$,}
\end{equation}
and $H= -\dfrac{i}{k} \mu^{-1} (\nabla \times E - K)$.
\end{lem}
\begin{proof}
Let $(E,H)\in L^2(\Omega)^3\times L^2(\Omega)^3$ solve \eqref{eq:Maxwell_inhom1}--\eqref{eq:Maxwell_inhom2}. Then \eqref{eq:Maxwell_inhom1}
implies that
\[
E\in H(\cu,\Omega)\quad \text{ and } \quad H= -\frac{i}{k} \mu^{-1} (\nabla \times E - K),
\]
and combining \eqref{eq:Maxwell_inhom1} and \eqref{eq:Maxwell_inhom2} we obtain 
\begin{equation}
\label{eq:Maxwell_variational_curlcurl}
\nabla \times \left( \mu^{-1} \left( \nabla \times E - K\right) \right) -   k^2 \epsilon E=  ik J,
\end{equation}
which also shows that $\mu^{-1} \left( \nabla \times E - K\right)\in H(\mathrm{curl};\Omega)$. Using \eqref{eq:Maxwell_variational_curlcurl} and the integration by parts formula \eqref{eq:int_by_parts_tangential}, it follows that for all $F\in H_0(\mathrm{curl};\Omega)$
\begin{align*}
\int_\Omega i k J \cdot \overline F \dx &= \int_\Omega \left(\nabla \times \left( \frac{1}{\mu} \left( \nabla \times E - K\right) \right)\right)
\cdot \overline F \dx - \int_\Omega k^2 \epsilon E \cdot \overline F \dx\\
&= \int_\Omega \left( \mu^{-1} \left( \nabla \times E - K\right) \right)
\cdot \left(\nabla \times \overline F\right) \dx - \int_\Omega k^2 \epsilon E \cdot \overline F \dx,
\end{align*}
and thus \eqref{eq:Maxwell_inhom_variational_Hcurl} holds.

On the other hand, if $E\in H(\cu,\Omega)$ fulfills \eqref{eq:Maxwell_inhom_variational_Hcurl} for all $F\in H_0(\cu,\Omega)$, then 
this also holds for all $F\in C_0^\infty(\Omega)^3$ which (by the definition of distributional derivatives) shows that 
\begin{align*}
\nabla \times \left( \mu^{-1} \nabla \times E \right) - k^2 \epsilon E =  i k J + \nabla \times \left( \mu^{-1} K \right),
\end{align*}
and thus
\begin{align*}
\frac{1}{ik}\nabla \times \left( \mu^{-1} (\nabla \times E - K) \right) + i k \epsilon E=  J.
\end{align*}
Defining $H:= \dfrac{1}{ik} \mu^{-1}(\nabla \times E - K)$ it follows that $H\in H(\mathrm{curl};\Omega)$ and
that $E$ and $H$ solve \eqref{eq:Maxwell_inhom1}--\eqref{eq:Maxwell_inhom2}.
\end{proof}

If $H(\cu,\Omega)$ was compactly embedded in $L^2(\Omega)$, then Theorem \ref{thm:Maxwell_well_posedness} would immediately 
follow from Lemma~\ref{lemma:Maxwell_variational_H0curl} by Fredholm theory arguments. Unfortunately this is not the case, so that we require
an additional variational formulation on the space
\begin{align*}
\mathcal{H}&:=\{ E\in L^{2}(\Omega)^3:\ \nabla\times E\in L^{2}(\Omega)^3,
\ \nabla \cdot (\epsilon E)=0,\ \nu \times E|_{\partial \Omega}=0\},
\end{align*}
which is compactly embedded in $L^2(\Omega)^3$, by using the following inequality 
\begin{align*}
\|\nabla E\|_{L^2(\Omega)}\leq C \left\{\|\nabla \times E\|_{L^2(\Omega)}+\|\nabla \cdot (\epsilon E)\|_{L^2(\Omega)}+\|\nu\times E\|_{H^{1/2}(\partial \Omega)}\right\},
\end{align*}
for some positive constant $C$ depending only on $\epsilon $ and $\Omega$ (for example, see \cite[Theorem 2.3]{shen2014lp}).
We will now first consider the Maxwell system with homogeneous boundary data and divergence free electric currents, so that the solution 
lies in $\mathcal H$. After that we will show that the general Maxwell system can always be transformed (or gauged) to fulfill this condition.

\begin{lem}\label{lemma:Maxwell_variational_H}
For $f=0$ and $\nabla\cdot J=0$, 
$(E,H)\in L^2(\Omega)^3\times L^2(\Omega)^3$ solves \eqref{eq:Maxwell_inhom1}--\eqref{eq:Maxwell_inhom3}
if and only if $E\in \mathcal{H}$ solves
\begin{equation}
\label{eq:Maxwell_inhom_variational_H}
\mathcal B(E,F)= \int_\Omega ikJ\cdot \overline F\dx +  \int_\Omega  \left( \mu^{-1} K \right) \cdot \left(\nabla \times \overline F\right) \dx
\quad \text{for all $F\in \mathcal H$,}
\end{equation}
and $H= -\dfrac{i}{k} \mu^{-1} (\nabla \times E - K)$.
\end{lem}
\begin{proof}
If $(E,H)\in L^2(\Omega)^3\times L^2(\Omega)^3$ fulfill \eqref{eq:Maxwell_inhom1}--\eqref{eq:Maxwell_inhom3},
then clearly $E\in \mathcal{H}$ and Lemma \ref{lemma:Maxwell_variational_H0curl} shows that 
\eqref{eq:Maxwell_inhom_variational_H} is fulfilled for all $F\in \mathcal H \subset H_0(\cu,\Omega)$.

To prove the other direction, let $E\in \mathcal{H}$ fulfill \eqref{eq:Maxwell_inhom_variational_H} for all $F\in \mathcal H$. Given $\Phi\in H_0(\cu,\Omega)$, there exists a solution $\varphi\in H_0^1(\Omega)$ of 
\[
\nabla\cdot (\epsilon \nabla \varphi)=- \nabla\cdot (\epsilon\Phi)
\]
and thus $F:=\Phi+\nabla \varphi\in \mathcal H$. Using \eqref{eq:Maxwell_inhom_variational_H} it follows that for
\begin{align*}
\mathcal B (E,\Phi)&= \mathcal B (E,F) -\mathcal B (E,\nabla \varphi)\\
&=  \int_\Omega ikJ\cdot\overline F\dx +  \int_\Omega  \left( \mu^{-1} K \right) \cdot \left(\nabla \times \overline F\right) \dx  + \int_\Omega  k^2 \epsilon E \cdot \nabla \overline \varphi \dx\\
&= \int_\Omega ikJ\cdot \overline\Phi\dx +  \int_\Omega  \left( \mu^{-1} K \right) \cdot \left(\nabla \times \overline \Phi\right) \dx,
\end{align*} 
where we used $\nabla \times (\nabla \varphi)=0$, $\nabla\cdot J =0$, $\nabla\cdot (\epsilon E)=0$ and $\varphi|_{\partial \Omega}=0$.
\end{proof}

We recall that we call $k>0$ a resonance frequency, if the homogeneous Maxwell system \eqref{eq:Maxwell_inhom1}--\eqref{eq:Maxwell_inhom3} with $J=0$, $K=0$ and $f=0$ admits a non-trivial solution. 

\begin{lem}\label{lemma:Maxwell_solution_H}
If $k>0$ is not a resonance frequency, then for every
$J,K\in L^{2}(\Omega)^3$ with $\nabla \cdot J=0$ and $f=0$, there exists a unique solution $(E,H)\in L^2(\Omega)^3\times L^2(\Omega)^3$ of \eqref{eq:Maxwell_inhom1}--\eqref{eq:Maxwell_inhom3}, and the solution depends continuously on $J,K\in L^2(\Omega)^3$. Moreover, the set of resonance frequencies is discrete. 
\end{lem}
\begin{proof}
Lemma \ref{lemma:Maxwell_variational_H} yields that $(E,H)\in L^2(\Omega)\times L^2(\Omega)$ solves \eqref{eq:Maxwell_inhom1}--\eqref{eq:Maxwell_inhom3}
if and only if
\[
\left(\mathcal A + \mathcal K(k) \right) E = l,
\]
where $\mathcal A, \mathcal K(k):\ \mathcal H\to \mathcal H^*$ are defined by 
\begin{alignat*}{2}
&\langle \mathcal A E, F\rangle  :=\int_\Omega \left( \mu^{-1}  \nabla \times E \right)
\cdot \left(\nabla \times F\right) \dx + \int_\Omega E\cdot F \dx,
\quad && \text{ for all } E,F\in \mathcal H,\\
&\langle \mathcal K(k) E, F\rangle  := - \int_\Omega (1+k^2\epsilon) E\cdot  F \dx,
\quad && \text{ for all } E,F\in \mathcal H,
\end{alignat*}
and $l\in \mathcal H^*$ is defined by 
\[
\langle l,F\rangle=\int_\Omega ikJ\cdot F\dx +  \int_\Omega  \left( \mu^{-1} K \right) \cdot \left(\nabla \times   F\right) \dx,
\quad  \text{ for all } F\in \mathcal H,
\]
where $\mathcal H^*$ is the dual space of $\mathcal H$.

Then $\mathcal A$ is a coercive linear bounded operator and thus continuously invertible due to the Lax-Milgram theorem.
For every $k\in \C$, $\mathcal K(k):\ \mathcal H\to \mathcal H^*$ is a linear compact operator due to the compact imbedding of $\mathcal H$ into $L^2(\Omega)^3$. $l\in \mathcal H^*$ depends linearly and continously on $J,K\in L^2(\Omega)^3$.
It thus follows from the Fredholm alternative, that $\mathcal A + \mathcal K(k)$ is continuously invertible if 
it is injective, i.e., if $k>0$ is not a resonance frequency.

Moreover, $\mathcal K(k)$ depends analytically on $k$, and for $\widehat k:=i$, $\mathcal A + \mathcal K(\widehat k)$ is coercive and thus continuously invertible. Hence, it follows from the analytic Fredholm theorem, that the set of resonances is discrete.
\end{proof}

Now we extend this result to non-homogeneous boundary data $f$ and non-divergence-free currents $J$ and prove theorem \ref{thm:Maxwell_well_posedness}.

\emph{Proof of Theorem \ref{thm:Maxwell_well_posedness}.}
(a) follows from Lemma \ref{lemma:Maxwell_variational_H0curl}. (b) and the uniqueness of the solution of the Maxwell system is proven in Lemma~\ref{lemma:Maxwell_solution_H}. To prove existence of the solution, let $J,K\in L^{2}(\Omega)^3$ and $f\in H^{-1/2}(\mathrm{div}_{\partial \Omega},\partial \Omega )$. We define $E_f=\gamma_t^{-1}f\in H(\cu,\Omega)$, i.e., $\nu \times E_f|_{\partial \Omega}=f$ and $E_f$ depends continuously and linear on $f$. Moreover let $\psi\in H_0^1(\Omega)$ solve
\[
\nabla \cdot (ik\epsilon \nabla \psi)= \nabla \cdot (J-ik\epsilon E_f),
\]
which also depends continuously and linearly on $J$ and $E_f$.

It follows from Lemma~\ref{lemma:Maxwell_solution_H} that there exists a solution $(E_0,H)\in L^2(\Omega)^3\times L^2(\Omega)^3$ of the gauged system
\begin{equation*}
\begin{cases}
\nabla\times E_0-ik\mu H=K-\nabla\times E_f  &  \mbox{ in }\Omega,\\
\label{eq:Maxwell_inhom2_gauged} \nabla\times H+ik\epsilon E_0=J-ik\epsilon E_f -ik\epsilon \nabla \psi  & \mbox{ in }\Omega,\\
\nu\times E_0|_{\partial \Omega}=0,
\end{cases}
\end{equation*}
and $E_0$ and $H$ depends linearly and continuously on $K-\nabla\times E_f\in L^2(\Omega)^3$ and $J-ik\epsilon E_f -ik\epsilon \nabla \psi\in L^2(\Omega)^3$. Hence, $E:=E_0+E_f+\nabla \psi$ and $H$ solve \eqref{eq:Maxwell_inhom1}--\eqref{eq:Maxwell_inhom3} and depend linearly and continuously on $J,K\in L^{2}(\Omega)^3$ and $f\in H^{-1/2}(\mathrm{div}_{\partial \Omega},\partial \Omega )$.
\hfill $\Box$

\subsection{Unique continuation}

The \textit{unique continuation property} (UCP) is an important property
to study the localized fields for differential equations. The
UCP for the anisotropic Maxwell system was studied by \cite{leis2013initial,nguyen2012quantitative}, which can help us to construct localized electromagnetic fields.
\begin{defn}
We say that $(\epsilon,\mu)$ satisfies the UCP in $\Omega$ if $(E,H)\in H(\mathrm{curl};\Omega)\times H(\mathrm{curl};\Omega)$
is a solution of 
\begin{equation}
\begin{cases}
\nabla\times E-ik\mu H=0 & \mbox{ in }\Omega,\\
\nabla\times H+ik\epsilon E=0 & \mbox{ in }\Omega,
\end{cases}\label{Maxwell without boundary}
\end{equation}
which vanishes in a nonempty open set $D$ in $\Omega$,
then $(E,H)\equiv(0,0)$ in $\Omega$. \end{defn}

\begin{thm}[Unique continuation property]\label{Theorem UCP}
The following properties hold: 
\begin{itemize} 
\item [(a)] Let $\epsilon,\mu\in L^\infty(\Omega,\R^{3\times 3})$ be piecewise Lipschitz continuous matrix-valued functions in $\Omega$ satisfying \eqref{Positive lower bound}. Suppose
that $(E,H)\in L^2(\Omega)^3\times L^2(\Omega)^3$ solves \eqref{Maxwell without boundary}. If $(E,H)=(0,0)$
in some nonempty open subset $D\subset\Omega$, we have $(E,H)=(0,0)$
in $\Omega$.

\item [(b)] Let $\mathcal{F}$ be a closed set in $\Omega$ such
that $\Omega\setminus \mathcal{F}$ is connected to a relatively open boundary part $\Gamma\subseteq\partial\Omega$.
Let $(E,H)\in H(\mathrm{curl};\Omega)\times H(\mathrm{curl};\Omega)$
solve 
\[
\begin{cases}
\nabla\times E-ik\mu H=0 & \mbox{ in }\Omega\backslash\mathcal{F},\\
\nabla\times H+ik\epsilon E=0 & \mbox{ in }\Omega\backslash\mathcal{F}.
\end{cases}
\]
If $\nu\times E|_\Gamma=\nu\times H|_\Gamma=0$ on $\Gamma$, then $(E,H)= (0,0)$
in $\Omega\backslash\mathcal{F}$.
\end{itemize}
\end{thm}
\begin{proof}
(a) The UCP of the Maxwell system was proved by Leis  \cite{leis2013initial}
when $\epsilon,\mu$ are $C^2$ scalar functions. When $\epsilon,\mu$
are Lipschitz continuous anisotropic parameters, the  UCP was proved
by Nguyen-Wang \cite{nguyen2012quantitative}. In \cite{liu2016mosco}, Liu, Rondi and Xiao have shown that the UCP holds for piecewise Lipschitz continuous matrix-valued functions $\epsilon$ and $\mu$ (see \cite[Section 2]{liu2016mosco}). We omit the proof here.

(b) Let $\mathcal{O}$ be a nonempty open set in $\mathbb{R}^{3}$
such that $\mathcal{O}\cap\partial\Omega=\Gamma$ and $\mathcal{F}\subset\mathcal{O}$.
In the set $\Omega\cup\mathcal{O}$, we define 
\[
\widehat{\epsilon}:=\begin{cases}
\epsilon & \mbox{ in }\Omega\\
1 & \mbox{ in }\mathcal{O}\backslash\Omega
\end{cases}\mbox{ and }\widehat{\mu}=\begin{cases}
\mu & \mbox{ in }\Omega\\
1 & \mbox{ in }\mathcal{O}\backslash\Omega
\end{cases}.
\]
Since $\nu\times E=\nu\times H=0$ on $\Gamma$, we can extend $(E,H)$
by $(0,0)$ and define the extension functions 
\[
\widehat{E}:=\begin{cases}
E & \mbox{ in }\Omega\\
0 & \mbox{ in }\mathcal{O}\backslash\Omega
\end{cases}\mbox{ and }\widehat{H}:=\begin{cases}
H & \mbox{ in }\Omega\\
0 & \mbox{ in }\mathcal{O}\backslash\Omega
\end{cases}.
\]

First, we prove that $(\widehat{E},\widehat{H})\in H(\mathrm{curl};\Omega\cup\mathcal{O})\times H(\mathrm{curl};\Omega\cup\mathcal{O})$.
For any $\phi\in C_{c}^{\infty}(\Omega\cup\mathcal{O})$, we have
\begin{eqnarray*}
	\int_{\Omega\cup\mathcal{O}}\widehat{E}\cdot(\nabla\times\phi)dx & = & \int_{\Omega}E\cdot(\nabla\times\phi)dx\\
	& = & \int_{\Omega}\left(\nabla\times E\right)\cdot\phi dx+\int_{\partial\Omega} E\cdot(\nu\times\phi) dS\\
	& = & \int_{\Omega\cup \mathcal{O}}\left(\left(\nabla\times E\right)\chi_\Omega\right)\cdot\phi dx,
\end{eqnarray*}
where we have used $E\cdot(\nu\times\phi)=-\phi\cdot\left(\nu\times E\right)=0$
on $\Gamma$ and $\phi=0$ on $\partial\Omega\backslash\Gamma$. This shows that $\widehat{E} \in H(\mathrm{curl};\Omega\cup\mathcal{O})$
and that $\nabla \times \widehat{E}$ is the zero extension of $\nabla \times E$. The same holds for $\widehat{H}$, and thus it also follows that $(\widehat{E},\widehat{H})$ is a solution of 
\[
\begin{cases}
\nabla\times\widehat{E}-ik\widehat{\mu}\widehat{H}=0 & \mbox{ in }\left(\Omega\cup\mathcal{O}\right)\backslash\mathcal{F},\\
\nabla\times\widehat{H}+ik\widehat{\epsilon}\widehat{E}=0 & \mbox{ in }\left(\Omega\cup\mathcal{O}\right)\backslash\mathcal{F}.
\end{cases}
\]
Notice that $\widehat{\epsilon}$ and $\widehat{\mu}$ are piecewise Lipschitz continuous functions fulfilling the ellipticity condition \eqref{Positive lower bound} and recall that $\widehat{E}=\widehat{H}=0$ in $\mathcal{O}\backslash\overline{\Omega}$
(a nonempty open set). Then by using (a), the UCP gives $\widehat{E}\equiv\widehat{H}\equiv0$
in $\left(\Omega\cup\mathcal{O}\right)\backslash\mathcal{F}$ and
completes the proof.
\end{proof}

\section{Localized electromagnetic fields\label{Section Localized fields}}

We will now present the main result on localizing and concentrating electromagnetic field.
We show that there exists boundary data (supported on an arbitrarily small boundary part) 
which generates an electromagnetic field with arbitrarily high energy on one part of the considered domain and arbitrarily small
energy on another part. This extends the related results in \cite{gebauer2008localized}
for the conductivity equation and \cite{harrach2017monotonicity}
for the Helmholtz equation to the more practical and challenging Maxwell system. In this section, we
prove the existence of localized fields by using the functional analysis techniques from \cite{gebauer2008localized}. Recall our main result as follows.

\begin{thm}
\label{Main Theorem Localization} Let $\Omega\subset\mathbb{R}^{3}$ be a bounded Lipschitz domain and $\Gamma\subseteq\partial \Omega$ be
a relatively open piece of the boundary. Let $\epsilon,\mu\in L^\infty(\Omega,\R^{3\times 3})$ be real-valued, piecewise Lipschitz continuous functions satisfying \eqref{Positive lower bound} and $k\in\mathbb{R}_+$ be
a non-resonant wavenumber. Let $D\Subset\Omega$ be a closed set with connected complement $\Omega \setminus D$.
For every open set $M\subseteq \Omega$ with $M\not\subseteq D$ (see Figure \ref{fig:1} for the schematic illustration), 
there exists a sequence $\left\{ f^{(\ell)}\right\} _{\ell\in\mathbb{N}}\subset C_c^\infty(\Gamma)$
such that the electromagnetic fields fulfill
\begin{equation}\label{eq:localized}
\int_{M}\left(|E^{(\ell)}|^{2}+|H^{(\ell)}|^2\right)\dx\to\infty\quad \mbox{ and }\quad \int_{D}\left(|E^{(\ell)}|^{2}+|H^{(\ell)}|^2\right)\dx\to0\mbox{ as }\ell\to\infty,
\end{equation}
where, for $\ell\in\mathbb{N}$, $(E^{(\ell)},H^{(\ell)})\in L^2(\Omega)^3\times L^2(\Omega)^3$
is a solution of 
\begin{alignat}{2}
\label{eq:Maxwell_locpot_1} \nabla\times E^{(\ell)}-ik\mu H^{(\ell)}&=0  && \quad \mbox{ in }\Omega,\\
\label{eq:Maxwell_locpot_2} \nabla\times H^{(\ell)}+ik\epsilon E^{(\ell)}&=0  && \quad \mbox{ in }\Omega,
\end{alignat}
with boundary data
\begin{align}\label{eq:Maxwell_locpot_3}
\nu\times E^{(\ell)}|_{\partial \Omega}=\begin{cases}
f^{(\ell)} & \mbox{ on }\Gamma,\\
0 & \mbox{ otherwise.}
\end{cases}
\end{align}
\end{thm}

\textbf{Proof of Theorem \ref{Main Theorem Localization}.}
We first note that it suffices to prove the theorem for an open subset of $M$. Hence, without loss of generality, we 
can assume that $M\Subset \Omega$ is open, $\overline M\cap D=\emptyset$ and that $\Omega\setminus (\overline M\cup D)$ is connected.
We follow the localized potentials strategy in \cite{gebauer2008localized,harrach2010exact,harrach2017monotonicity} and 
first describe the energy terms in Theorem \ref{Main Theorem Localization} as operator evaluations. Then we will show that the ranges of the adjoints of these operators 
have trivial intersection. A functional analytic relation between the norm of an operator evaluation and the range of its adjoint will then yield that
the operator evaluations cannot be bounded by each other, which then shows that we can drive one
energy term in theorem \ref{Main Theorem Localization} to infinity and the other one to zero.

For a measurable subset $O\subseteq \Omega$, we define
\[
\mathcal{L}_{O}:\ H(\Gamma) \to L^{2}(O)^3\times L^2(O)^3 \mbox{ by }f\mapsto (E|_O,H_O),
\]
where $H(\Gamma)$ was defined in \eqref{abstract trace space}, and $(E,H)\in L^2(\Omega)^3\times L^2(\Omega)^3$
solve \eqref{eq:Maxwell_locpot_1}--\eqref{eq:Maxwell_locpot_2}
with boundary data $\nu\times E|_{\partial \Omega}=f$. Now we characterize the adjoint of this operator.

\begin{lem}\label{lemma:adjoint_L}
The adjoint of $\mathcal{L}_O$ is given by
\[
\mathcal{L}_{O}^*:\ L^{2}(O)^3\times L^{2}(O)^3\to H(\Gamma)^*\mbox{ by }(J,K)\to -\nu\times(\widetilde{H}\times\nu) |_{\Gamma},
\]
where $(\widetilde{E},\widetilde{F})\in L^2(\Omega)^3\times L^2(\Omega)^3$ solve the (adjoint) Maxwell system (cf.\ Theorem \ref{thm:Maxwell_well_posedness})
\begin{equation*}
\begin{cases}
\nabla\times \widetilde{E}+ik\mu \widetilde{H}=K\chi_O & \mbox{ in }\Omega,\\
\nabla\times \widetilde{H}-ik\epsilon \widetilde{E}=J\chi_O  &  \mbox{ in }\Omega,\\
\nu\times \widetilde{E}|_{\partial \Omega}=0,
\end{cases}
\end{equation*}
and $K\chi_O$ and $J\chi_O$ denote the zero extensions of $K$ and $J$ to $\Omega$.
\end{lem}
\begin{proof}
As in subsection \ref{subsect:spaces_traces}, we write the dual pairing on $H(\Gamma)^*\times H(\Gamma)$ as an 
integral for the sake of readability. With this notation we have that
\begin{align*}
\lefteqn{\int_\Gamma \overline{f} \cdot \mathcal{L}_{O}^* (J, K)\dx[s]}\\
&= \int_O (J,K)\cdot \overline{(E|_O,H_O)} \dx = \int_\Omega J\chi_O \cdot \overline{E} \dx + \int_\Omega K\chi_O\cdot \overline{H}\dx\\
&= \int_\Omega J \chi_O \cdot \overline{E} \dx - \frac{1}{ik}  \int_\Omega K\chi_O\cdot \left( \mu^{-1} \nabla\times \overline{E}\right)\dx\\
&= \int_\Omega \left(\nabla\times \widetilde{H}-ik\epsilon \widetilde{E}\right)\cdot \overline{E}\dx
 - \frac{1}{ik}  \int_\Omega \left( \nabla\times \widetilde{E}+ik\mu \widetilde{H}\right)\cdot \left( \mu^{-1} \nabla\times \overline{E}\right)\dx\\
&= \int_\Omega \left(\nabla\times \widetilde{H}\right)\cdot \overline{E}\dx
-  \int_\Omega  \widetilde{H}\cdot \left(  \nabla\times \overline{E}\right)\dx\\
&\quad - \frac{1}{ik}\left( 
 \int_\Omega  \left( \mu^{-1} \nabla\times \overline{E}\right) \cdot \left( \nabla\times \widetilde{E}\right)  \dx
- \int_\Omega k^2 \epsilon \overline{E} \cdot \widetilde{E} \dx
\right)\\
&=  - \int_{\partial \Omega} \left(\nu \times \overline{E}|_{\partial \Omega}\right) \cdot  \left(\nu \times \widetilde{H}|_{\partial \Omega}\times \nu\right) \dx[S]= - \int_{\Gamma} \overline{f}\cdot  \left(\nu \times \widetilde{H}|_{\Gamma}\times \nu\right) \dx[S],
\end{align*}
where we used that $\epsilon$ and $\mu$ are real-valued and symmetric. We also utilized the integration by parts formula
\eqref{eq:int_by_parts_tangential} and that $\widetilde{E}\in H_0(\cu,\Omega)$ implies that
\[
\int_\Omega  \left( \mu^{-1} \nabla\times \overline{E}\right) \cdot \left( \nabla\times \widetilde{E}\right)  \dx
- \int_\Omega k^2 \epsilon \overline{E} \cdot \widetilde{E} \dx=0
\]
by Theorem~\ref{thm:Maxwell_well_posedness} (a).
\end{proof}

Then we have the following property for the ranges of the adjoint operators $\mathcal{L}_{M}^*$ and $\mathcal{L}_{D}^*$.
\begin{lem}
\label{Lemma funal anal 1} 
$\mathcal{L}_{M}$ and $\mathcal{L}_{D}^{*}$ are injective, 
the ranges $\mathscr{R}(\mathcal{L}_{M}^{*})$ and $\mathscr{R}(\mathcal{L}_{D}^{*})$ are both dense in $H(\Gamma)^*$, 
and 
\begin{equation}
\mathscr{R}(\mathcal{L}_{M}^{*})\cap\mathscr{R}(\mathcal{L}_{D}^{*})=\{0\}.\label{Zero intersection}
\end{equation}
\end{lem}
\begin{proof}
The proof follows from the UCP for the Maxwell system. By Theorem
\ref{Theorem UCP} (a), one can see that $\mathcal{L}_{M}$
and $\mathcal{L}_{D}$ are injective, and therefore $\mathscr{R}(\mathcal{L}_{M\cap B}^{*})$
and $\mathscr{R}(\mathcal{L}_{D}^{*})$ both are dense in $H(\Gamma)^*$.

To prove \eqref{Zero intersection}, let $g\in\mathscr{R}(\mathcal{L}_{M}^{*})\cap\mathscr{R}(\mathcal{L}_{D}^{*})$,
then there exist $J_{M},K_M\in L^{2}(M)$ and $J_{D},K_D\in L^{2}(D)$ such
that the solutions $(E_M,H_M), (E_D,H_D)\in L^2(\Omega)^3\times L^2(\Omega)^3$ of
\[
\begin{cases}
\nabla\times E_{M}+ik\mu H_{M}=K_M\chi_M & \mbox{ in }\Omega\\
\nabla\times H_{M}-ik\epsilon E_{M}=J_{M}\chi_{M} & \mbox{ in }\Omega\\
\nu\times E_{M}|_{\partial \Omega}=0 & \mbox{ on }\partial\Omega
\end{cases}\mbox{ and }\begin{cases}
\nabla\times E_{D}+ik\mu H_{D}=K_D\chi _D & \mbox{ in }\Omega\\
\nabla\times H_{D}-ik\epsilon E_{D}=J_{D}\chi_{D} & \mbox{ in }\Omega\\
\nu\times E_{D}|_{\partial \Omega}=0 & \mbox{ on }\partial\Omega
\end{cases}
\]
fulfill 
\[
\nu\times (H_{M}\times \nu)|_\Gamma=g=\nu\times (H_{D}\times \nu)|_\Gamma.
\]

Since $\Omega\setminus (\overline M \cup D)$ is connected, we obtain by using Theorem \ref{Theorem UCP} (b) that
\begin{equation*}
E_{M}=E_{D}\mbox{ in }\Omega\setminus (\overline{M}\cup D).\label{some unique solution}
\end{equation*}
Hence, we can define
\[
\mathbb{E}:=\begin{cases}
E_{D} & \mbox{ in }M\\
E_{M} & \mbox{ in }D\\
E_{D}=E_{M} & \mbox{ in } \Omega\setminus (\overline{M}\cup D)
\end{cases},\mbox{ and }\mathbb{H}:=\begin{cases}
H_{D} & \mbox{ in }M\\
H_{M} & \mbox{ in }D\\
H_{D}=H_{M} & \mbox{ in } \Omega\setminus (\overline{M}\cup D)
\end{cases}.
\]
As in the proof of UCP,  Theorem \ref{Theorem UCP} (b), it is easy
to see that $(\mathcal{E},\mathcal{H})\in L^2(\Omega)^3\times L^2(\Omega)^3$
is a solution of 
\[
\begin{cases}
\nabla\times\mathbb{E}+ik\mu\mathbb{H}=0 & \mbox{ in }\Omega,\\
\nabla\times\mathbb{H}-ik\epsilon\mathbb{E}=0 & \mbox{ in }\Omega,\\
\nu\times\mathbb{E}|_{\partial \Omega}=0 & \mbox{ on }\partial\Omega.
\end{cases}
\]
Since $k$ is non-resonant, it follows that $(\mathbb{E},\mathbb{H})=(0,0)$ so that $g=0$. This completes the proof.\end{proof}

Now we can use the following tool from functional analysis.

\begin{lem}\label{Lemma funal anal 2}
Let $X$, $Y_1$ and $Y_2$ be Hilbert spaces, and $\mathcal{A}_1:\ X\to Y_1$ and $\mathcal{A}_2:\ X\to Y_2$ be linear bounded operators.
Then
\[
\exists C>0:\ \norm{\mathcal{A}_1x}\leq C \norm{\mathcal{A}_2 x} \quad \forall x\in X
\quad \text{ if and only if } \quad \mathscr{R}(\mathcal{A}_1^*)\subseteq \mathscr{R}(\mathcal{A}_2^*). 
\] 
\end{lem}
\begin{proof}
This is proven for reflexive Banach spaces in \cite[Lemma~2.5]{gebauer2008localized}. 
\end{proof}

\begin{proof}[Proof of Theorem \ref{Main Theorem}]
From Lemma \ref{Lemma funal anal 1} it follows that $\mathscr{R}(\mathcal{L}_{M}^{*})\not\subseteq \mathscr{R}(\mathcal{L}_{D}^{*})$. Using 
Lemma \ref{Lemma funal anal 2} this shows that
\begin{equation*}
\not\exists C>0:\ \norm{\mathcal{L}_{M}f}\leq C \norm{\mathcal{L}_{D}f}
\quad \text{ for all } f\in H(\Gamma),
\end{equation*}
and by continuity of $\mathcal{L}_{M}$ and $\mathcal{L}_{D}$ and density of $C^\infty_c(\Gamma)\subset H(\Gamma)$ this is equivalent to
\begin{equation}\label{eq:norm_bound_Cinfty}
\not\exists C>0:\ \norm{\mathcal{L}_{M}f}\leq C \norm{\mathcal{L}_{D}f}
\quad \text{ for all } f\in C^\infty_c(\Gamma).
\end{equation}
Using \eqref{eq:norm_bound_Cinfty} with $C:=\ell^2$ for all $\ell\in \mathbb{N}$ we thus obtain a sequence $\left\{\widetilde f^{(\ell)}\right\}_{\ell\in \mathbb{N}}\subset C^\infty_c(\Gamma)$ 
with 
\[
\norm{\mathcal{L}_{M}\widetilde f^{(\ell)}}> \ell^2 \norm{\mathcal{L}_{D}\widetilde f^{(\ell)}} \quad \text{ for all } \ell\in \mathbb{N}.
\] 
By injectivity $\mathcal{L}_{M}\widetilde f^{(\ell)}\neq 0$ implies $\widetilde f^{(\ell)}\neq 0$ and $\mathcal{L}_{D}\widetilde f^{(\ell)}\neq 0$, so that we can define
\[
f^{(\ell)}:=\dfrac{\widetilde f^{(\ell)}}{\ell \norm{\mathcal{L}_{D}\widetilde f^{(\ell)}}} \in C^\infty_c(\Gamma)
\]
and it follows that
\begin{align*}
\int_{M}\left(|E^{(\ell)}|^{2}+|H^{(\ell)}|^2\right)\dx &= \norm{\mathcal{L}_{M} f^{(\ell)}}^2>\ell^2\to \infty,\\
\int_{D}\left(|E^{(\ell)}|^{2}+|H^{(\ell)}|^2\right)\dx &= \norm{\mathcal{L}_{D} f^{(\ell)}}^2=\frac{1}{\ell^2}\to 0,
\end{align*}
so that Theorem \ref{Main Theorem} is proven.
\end{proof}


\begin{rem}
\begin{enumerate}[(a)]
\item A constructive version of the existence proof for the localized fields can be obtained as in \cite[Lemma 2.8]{gebauer2008localized}.
\item With the same arguments as in \cite[Section~4.1]{harrach2017monotonicity} one can also show that
for all spaces $W\subseteq H(\Gamma)$ with finite codimension, one can find a sequence $\left\{f^{(\ell)}\right\}_{\ell \in \mathbb N}\subset W$ 
so that the corresponding electromagnetic fields fulfill \eqref{eq:localized}. This might be useful for developing monotonicity-based
reconstruction methods as in \cite{harrach2017monotonicity}.
\end{enumerate}
\end{rem}

\section{Runge approximation property for the partial data Maxwell system \label{Section 4}}

In this section we derive an extension of the localization result in Theorem \ref{Main Theorem} which shows a \emph{Runge approximation property} for the partial data Maxwell system that we consider of independent interest.
We show that every solution of the Maxwell system on a subset of $\Omega$ with Lipschitz boundary and connected complement can be approximated arbitrarily well by a sequence of solutions on the whole domain $\Omega$ with partial boundary data. Since we can choose a solution that is zero on a part of $\Omega$ and non-zero on another part of $\Omega$, this also implies a-fortiori the localization result Theorem \ref{Main Theorem}, cf.\ also \cite{harrach2017monotonicity} for the connection between Runge approximation properties and localized solutions.

\begin{thm}\label{Main Theorem Runge}
Let $\Omega\subset\mathbb{R}^{3}$ be a bounded Lipschitz domain and $\Gamma\subseteq\partial \Omega$ be a relatively open piece of the boundary. Let $\epsilon,\mu\in L^\infty(\Omega,\R^{3\times 3})$ be real-valued, piecewise Lipschitz continuous functions satisfying \eqref{Positive lower bound} and $k\in\mathbb{R}_+$ be
a non-resonant wavenumber. 

Let $O\Subset\Omega$ be an open set with Lipschitz boundary and connected complement $\Omega\setminus \overline O$. For every solution $(e,h)\in L^2(O)^3\times L^2(O)^3$
of
\begin{alignat}{2}
\label{eq:Maxwell_Runge_subset_1} \nabla\times e-ik\mu h&=0  && \quad \mbox{ in } O,\\
\label{eq:Maxwell_Runge_subset_2} \nabla\times h+ik\epsilon e&=0  && \quad \mbox{ in } O,
\end{alignat}
there exists a sequence $\left\{ f^{(\ell)}\right\} _{\ell\in\mathbb{N}}\subset C_c^\infty(\Gamma)$
such that the electromagnetic fields fulfill
\begin{align*}
\norm{E^{(\ell)}-e}_{L^2(O)}\to 0 \quad \text{ and } \quad \norm{H^{(\ell)}-h}_{L^2(O)}\to 0 \quad \text{as}\quad\ell\to \infty,
\end{align*}
where $(E^{(\ell)},H^{(\ell)})\in L^2(\Omega)^3\times L^2(\Omega)^3$
solve \eqref{eq:Maxwell_locpot_1}--\eqref{eq:Maxwell_locpot_3}.
\end{thm}
\begin{proof}
Let $(e,h)\in L^2(O)^3\times L^2(O)^3$ solve \eqref{eq:Maxwell_Runge_subset_1}--\eqref{eq:Maxwell_Runge_subset_2}.
With the operator $\mathcal L_O$ introduced in Section \ref{Section Localized fields},
we will show that
\begin{align}\label{eq:Runge_ef_property}
(e,h)\in \overline{\mathscr{R}(\mathcal L_O)}=\left(\mathscr{R}(\mathcal L_O)^\perp\right)^\perp
=\mathscr{N}(\mathcal L_O^*)^\perp,
\end{align}
where closure and orthogonality are understood with respect to the $L^2(O)^3\times L^2(O)^3$-scalar product.
This shows that the assertion holds with a sequence $\left\{ f^{(\ell)}\right\} _{\ell\in\mathbb{N}}\subset H(\Gamma)$, and
it follows by density that the assertion holds with a sequence $\left\{ f^{(\ell)}\right\} _{\ell\in\mathbb{N}}\subset C_c^\infty(\Gamma)$.

To prove \eqref{eq:Runge_ef_property} let $(J,K)\in \mathscr{N}(\mathcal L_O^*)\subseteq L^2(O)^3\times L^2(O)^3$.
Then, by Lemma \ref{lemma:adjoint_L}, there exist $(E,H)\in L^2(\Omega)^3\times L^2(\Omega)^3$ that solve 
\begin{equation*}
\begin{cases}
\nabla\times E+ik\mu H=K\chi_O & \mbox{ in }\Omega,\\
\nabla\times H-ik\epsilon E=J\chi_O  &  \mbox{ in }\Omega,\\
\nu\times E|_{\partial \Omega}=0 & \mbox{ on }\partial \Omega,
\end{cases}
\end{equation*}
with $\nu\times(H\times\nu) |_{\Gamma}=\mathcal L_O^* (J,K)=0$. The unique continuation property in Theorem~\ref{Theorem UCP}
implies that $(E,H)=(0,0)$ on $\Omega\setminus \overline{O}$ and thus
\[
\nu \times E|_{\partial O}=0=\nu \times H|_{\partial O}.
\]
Hence, using the integration by parts formula \eqref{eq:int_by_parts_tangential}, it follows that
\begin{align*}
\int_O \left( e\cdot \overline{J} + h\cdot \overline{K}\right) \dx
&= \int_O  e \cdot \left( \nabla\times \overline{H}+ik\epsilon \overline{E} \right) \dx
+ \int_O h \cdot \left( \nabla\times \overline{E}-ik\mu \overline{H} \right) \dx\\
&= \int_O  \left( ( \nabla \times e) \cdot \overline{H}+ik \epsilon e \cdot \overline{E} \right) \dx
+ \int_O  \left( \nabla \times h \cdot \overline{E}-ik\mu h\cdot \overline{H} \right) \dx\\
&= \int_O   ( \nabla \times e -ik\mu h ) \cdot \overline{H} \dx
+ \int_O  (\nabla \times h +ik \epsilon e )\cdot \overline{E}  \dx=0.
\end{align*}
This shows $(e,h)\perp (J,K)$ so that \eqref{eq:Runge_ef_property} holds, and thus the assertion is proven. 
\end{proof}

\begin{rem}
The Runge approximation property in Theorem~\ref{Main Theorem Runge} implies the localization property in Theorem~\ref{Main Theorem} by the following argument. Let $D\Subset\Omega$ be a closed set with connected complement $\Omega \setminus D$.
and $M\subseteq \Omega$ be an open set with $M\not\subseteq D$ (see again Figure \ref{fig:1}).
By shrinking $M$ and enlarging $D$, we can assume that $M$ is a open set and $D$ is a closed sets with Lipschitz boundaries,
$\overline M \cup D\Subset \Omega$,  
$\overline M\cap D=\emptyset$ and that $\Omega\setminus (\overline M\cup D)$ is connected.

The unique continuation property in Theorem~\ref{Theorem UCP} implies that a solution of the Maxwell system in $\Omega$ 
with non-trivial boundary data cannot vanish identically on $M$. This shows that there exists a non-zero solution
of the Maxwell system on $M$. We extend this solution by zero on $D$, and obtain a solution $(e,h)\in L^2(O)\times L^2(O)$
on $O:=M\cup \mathrm{int}D$ with $(e,h)|_{\mathrm{int}D}\equiv (0,0)$ and $(e,h)|_M\not\equiv (0,0)$. Then the Runge approximation sequence from 
Theorem~\ref{Main Theorem Runge} converges to zero on $D$ but not on $M$ and a simple scaling argument as in the proof
of Theorem~\ref{Main Theorem} in Section \ref{Section Localized fields} gives a sequence of electromagnetic fields with
\begin{equation*}
\int_{M}\left(|E^{(\ell)}|^{2}+|H^{(\ell)}|^2\right)\dx\to\infty\quad \mbox{ and }\quad \int_{D}\left(|E^{(\ell)}|^{2}+|H^{(\ell)}|^2\right)\dx\to0\mbox{ as }\ell\to\infty,
\end{equation*}
and $\nu \times E^{(\ell)}\in C_c^\infty(\Gamma)$ for all $\ell \in \mathbb N$, which also proves Theorem~\ref{Main Theorem}.
\end{rem}


\textbf{}\\
\textbf{Acknowledgments.} H. Liu was supported by the FRG and
startup grants from Hong Kong Baptist University, and Hong Kong RGC
General Research Funds, 12302415 and 12302017.

\bibliographystyle{plain}
\bibliography{ref}

\end{document}